\documentclass{amsart}
\usepackage{xspace,bbm}
\usepackage[small]{diagrams}

\def\makedef#1#2{\expandafter\gdef\csname#1\endcsname{#2}}
\def\makelet#1#2{\expandafter\let\csname#1\expandafter\endcsname\csname#2\endcsname}
\def\mat#1{\ensuremath{#1}\xspace}
\def\defmath#1#2{\makedef{#1}{\mat{#2}}}
\def\redefmath#1{\makelet{temp@#1}{#1}\defmath{#1}{\csname temp@#1\endcsname}}

\def\defbb#1{\defmath{c#1}{\mathbb{#1}}}		
\def\defcal#1{\defmath{l#1}{\mathcal{#1}}}	
\def\deffrak#1{\defmath{g#1}{\mathfrak{#1}}} 
\def\bbcal#1{\defbb{#1}\defcal{#1}\deffrak{#1}}

\bbcal A
\bbcal B
\bbcal C
\bbcal D
\bbcal E
\bbcal F
\bbcal G
\bbcal H
\bbcal I
\bbcal J
\bbcal K
\bbcal L
\bbcal M
\bbcal N
\bbcal O
\bbcal P
\bbcal Q
\bbcal R
\bbcal S
\bbcal T
\bbcal U
\bbcal V
\bbcal W
\bbcal X
\bbcal Y
\bbcal Z

\deffrak b
\deffrak c
\deffrak C
\deffrak d
\deffrak h
\deffrak j
\deffrak k

\deffrak m
\deffrak n
\deffrak p
\deffrak q
\deffrak r
\deffrak s
\deffrak t
\deffrak u
\deffrak v

\def\al{\mat{\alpha}}

\def\ga{\mat{\gamma}}

\def\Eps{\mat{\lE}}

\def\hi{\mat{\chi}}

\def\la{\mat{\lambda}}

\def\Om{\mat{\Omega}}

\def\vi{\mat{\varphi}}

\redefmath{eta}
\redefmath{xi}
\redefmath{mu}
\redefmath{nu}
\redefmath{Phi}
\redefmath{Psi}
\redefmath{psi}
\redefmath{pi}
\redefmath{rho}

\def\mrm@#1{\mat{\mathrm{#1}}}

\def\opr#1{\operatorname{#1}}
\def\DMO#1#2{\defmath{#1}{\opr{#2}}}
\def\oper#1{\defmath{#1}{\opr{#1}}}

\oper{asd}

\DMO{Hom}{Hom}
\DMO{RHom}{RHom}
\DMO{lHom}{\lH\mathit{om}}
\DMO{Ext}{Ext}
\DMO{lExt}{\lE\mathit{xt}}
\DMO{End}{End}
\DMO{Aut}{Aut}
\DMO{Fun}{Fun}
\DMO{Tor}{Tor}
\DMO{ext}{ext}

\DMO{Ob}{Ob}
\DMO{Mor}{Mor}
\DMO{im}{im}
\DMO{coim}{coim}
\DMO{coker}{coker}
\DMO{Arr}{Arr}

\DMO{Id}{Id}
\DMO{id}{id}
\DMO{add}{add} 
\DMO{ind}{ind} 
\DMO{pro}{pro} 
\DMO{Map}{Map} %
\DMO{Iso}{Iso} %
\DMO{Isom}{Isom}%
\DMO{Ind}{Ind}
\DMO{Bun}{Bun}
\DMO{Higgs}{Higgs}
\DMO{Hitch}{Hitch}

\DMO{Loc}{Loc}
\DMO{Maps}{Maps}
\makeatletter
\@ifpackageloaded{bbm}
{}
{}
\makeatother

\DMO{Presh}{Presh}

\DMO{coalg}{Coalg}
\DMO{Rep}{Rep}
\DMO{Cor}{Cor}

\DMO{Mod}{Mod}
\DMO{rad}{rad}
\DMO{soc}{soc}
\DMO{ann}{ann}
\DMO{pd}{pd}

\DMO{Spec}{Spec}
\DMO{Specm}{Specm}
\DMO{Max}{Max}
\DMO{spec}{Spec}
\DMO{Proj}{Proj}
\DMO{supp}{supp}
\DMO{Coh}{Coh}
\DMO{coh}{coh}
\DMO{Qcoh}{QCoh}
\DMO{QCoh}{QCoh}
\DMO{Pic}{Pic}
\DMO{Div}{Div}
\DMO{ch}{ch}
\DMO{Hilb}{Hilb}
\DMO{Fitt}{Fitt}
\DMO{Quot}{Quot}

\DMO{Gras}{Gr}
\DMO{Grass}{Gr}
\DMO{Flag}{Flag}
\DMO{Jac}{Jac}

\DMO{cone}{cone}
\DMO{Tw}{Tw}
\DMO{rank}{rk}
\DMO{rk}{rk}
\DMO{codim}{codim}
\DMO{cov}{cov}
\DMO{sgn}{sgn}
\DMO{td}{td}
\DMO{GL}{GL}
\DMO{PGL}{PGL}
\DMO{SL}{SL}
\DMO{SO}{SO}

\DMO{Der}{Der}
\DMO{der}{Der}
\DMO{coder}{Coder}
\DMO{diag}{diag}
\DMO{HMod}{HMod} 
\DMO{ad}{ad}
\DMO{Ad}{Ad}

\DMO{Ho}{Ho}
\def\Gr{\mathbf{Gr}} 

\DMO{har}{char}
\DMO{sk}{sk}
\DMO{cosk}{cosk}
\DMO{Gal}{Gal}
\DMO{tr}{tr}
\DMO{Tr}{Tr}
\DMO{Sh}{Sh}
\DMO{Is}{Is} 
\DMO{Hol}{Hol} 
\DMO{Lie}{Lie} 
\DMO{Res}{Res} 
\DMO{irr}{irr} %
\DMO{Irr}{Irr} %
\DMO{Exp}{Exp} %
\DMO{Log}{Log} %
\DMO{Pow}{Pow}
\DMO{pow}{pow}
\DMO{mult}{mult} %
\DMO{height}{ht} %
\DMO{wt}{wt}
\DMO{Vect}{Vect}
\DMO{moda}{mod}
\DMO{hd}{hd} 
\DMO{face}{face}
\DMO{Sym}{Sym}
\DMO{Com}{Com} 
\DMO{Eig}{Eig} 
\DMO{cl}{cl} 
\DMO{Li}{Li} 
\DMO{Imxx}{Im} 
\DMO{Rexx}{Re}

\DMO{St}{St}
\DMO{Var}{Var}

\def\n#1{\mat{\lvert#1\rvert}}
\def\lrb#1{\left(#1\right)}

\def\iso{\simeq}
\def\ts{\otimes}
\def\tl#1{\mat{\tilde{#1}}}
\def\wtl#1{\mat{\widetilde{#1}}}
\def\what#1{\mat{\widehat{#1}}}

\def\xx{\times}

\def\pser#1{[\![#1]\!]} 
\def\lser#1{\mbox{(\!(}#1\mbox{)\!)}} 

\def\half{\mat{\frac12}}
\def\oh{\half} 

\def\inv{^{-1}}

\def\sets#1#2{\mat{\{ #1 \mid #2\}}}

\def\emb{\hookrightarrow}
\def\mto{\mapsto}
\def\arr{\futurelet\test\arrtest}
\def\arrtest{\ifx^\test\let\next\arra\else\let\next\arrb\fi\next}
\def\arra^#1{\xrightarrow{#1}} \def\arrb{\to}

\def\arrowsD{
\def\mto{{\:\vrule height .9ex depth -.2ex width .04em\!\!\!\;\ar}}
\def\ar{{\:\vrule depth -.52ex height .60ex width 0.85em\;\!\!\rhla\,}}
\def\arr{\futurelet\test\arrtest}
\def\arrtest{\ifx^\test\let\next\arra\else\let\next\arrb\fi\next}
\def\arra^##1{\rTo^{##1}} \def\arrb{\ar}
\def\emb{\futurelet\test\embtest}
\def\embtest{\ifx^\test\let\next\emba\else\let\next\embb\fi\next}
\def\emba^##1{\rInto^{##1}} \def\embb{{\:\rthooka\!\!\!\ar}}
\newarrow{Eq}=====
\def\rrarr{\pile{\rTo\\ \rTo}}
\def\lrarr{\pile{\rTo\\ \lTo}}   
\newarrow{ShortTo}{}{}-->
}

\def\arrowsDStandard{
\newarrow{TeXto}----{->}
\newarrow{TeXinto}C---{->}
\newarrow{TeXonto}----{->>}
\newarrow{TeXdashto}{}{dash}{}{dash}{->}
\newarrow{Eq}=====
\def\ar{\rightarrow}
\def\emb{\futurelet\test\embtest}
\def\embtest{\ifx^\test\let\next\emba\else\let\next\embb\fi\next}
\def\emba^##1{\rTeXinto^{##1}} \def\embb{\hookrightarrow}
\def\rrarr{\pile{\rTo\\ \rTo}}
\def\lrarr{\pile{\rTo\\ \lTo}}   
}

\def\theorems{
\newcounter{nthr} 
\numberwithin{nthr}{section}
\let\theHnthr\thenthr
\newtheorem{thr}[nthr]{Theorem}
\newtheorem{prp}[nthr]{Proposition}
\newtheorem{lmm}[nthr]{Lemma}
\newtheorem{crl}[nthr]{Corollary}
\newtheorem{clm}[nthr]{Claim}
\newtheorem{conj}{Conjecture}
\newtheorem{rmr}[nthr]{Remark}
\theoremstyle{definition}
\newtheorem{dfn}[nthr]{Definition}
\newtheorem{exm}[nthr]{Example}
\newtheorem{claim}[nthr]{Claim}
\theoremstyle{remark}
}

\newif\ifrem\remtrue
\def\rem#1{\ifrem\marginpar{\raggedright\footnotesize #1}\fi}

\def\br{\linebreak}
\def\ub#1{\mat{\overline{#1}}}  

\def\ie{i.e.\ }
\def\eg{e.g.\ }
\def\cf{cf.\ }

\def\eprint#1#2{%
\expandafter\ifx\csname eprint@#1\endcsname\relax#1:#2%
\else\def\itemID{#2}\csname eprint@#1\endcsname\fi}

\def\defArchive#1#2{%
\makedef{eprint@#1}{#2}}

\defArchive{arxiv}{\href{http://arXiv.org/abs/\itemID}{{\tt arXiv:\itemID}}}
\defArchive{numdam}{\href{http://www.numdam.org/item?id=\itemID}{Numdam:\itemID}}

\theorems\arrowsDStandard
\usepackage{hyperref}

\begin{document}
\remfalse
\def\OO{\Om'}
\def\TT{T}
\title[ADHM recursion formula]{Solutions of the motivic ADHM recursion formula}%
\author{Sergey Mozgovoy}%
\email{mozgovoy@maths.ox.ac.uk}%
\begin{abstract}
We give an explicit solution of the ADHM recursion formula conjectured by Chuang, Diaconescu, and Pan. This solution is closely related to the formula for the Hodge polynomials of Higgs moduli spaces conjectured by Hausel and Rodriguez-Villegas. We solve also the twisted motivic ADHM recursion formula. As a byproduct we obtain a conjectural formula for the motives of twisted Higgs moduli spaces, which generalizes the conjecture of Hausel and Rodriguez-Villegas.
\end{abstract}

\maketitle

\section{Introduction}
The main goal of this paper is to solve the ADHM recursion formula conjectured by Chuang, Diaconescu, and Pan \cite{chuang_motivic}. We will show that in the untwisted case solutions are given by expressions closely related to the Hausel-Rodriguez-Villegas polynomials \cite[Conj.~5.6]{hausel_mirrora}, which are conjecturally the Hodge polynomials of the Higgs moduli spaces. This observation was made originally in \cite{chuang_motivic} for small rank and genus. To provide a solution in the twisted case we generalize the conjecture of Hausel and Rodriguez-Villegas to the moduli spaces of twisted Higgs bundles (see Conjecture \ref{E twisted}) and show that these new invariants give a solution of the ADHM recursion formula (see Theorem \ref{main1}).

The second goal of this paper is to understand the ADHM recursion formula in the framework of motivic Donaldson-Thomas invariants developed by Kontsevich and Soibelman \cite{kontsevich_stability}.
Both the ADHM recursion formula and the conjecture of Hausel and Rodriguez-Villegas are formulated using the Hodge polynomials for ordinary cohomologies. These invariants are not of motivic nature, but we can use the Poincar\'e duality to formulate both conjectures in terms of $E$-polynomials, which are motivic invariants. The new conjectures can be formulated actually in terms of motives. The reason is that the main ingredient in the Hausel-Rodriguez-Villegas formula (as well as in the ADHM recursion formula) can be identified with the $E$-polynomial specialization of the zeta-function of a curve, see Remark~\ref{rmr:zeta}. By taking the motivic zeta-function instead of this specialization, we obtain the motivic ADHM recursion formula (Conjecture \ref{conj:rec1}) as well as the motivic version of the Hausel-Rodriguez-Villegas conjecture (Conjecture \ref{E untwisted}).

Motivic ADHM recursion formula that we formulate in Conjecture \ref{conj:rec1} can be considered as a wall-crossing formula for framed objects \cite[Corollary 4.25]{mozgovoy_wall-crossing}.
But it is still conjectural because we don't have yet an integration map for ADHM sheaves, which is needed to apply the results from~\cite{mozgovoy_wall-crossing}. The original ADHM recursion formula \cite{chuang_motivic} was motivated by the wall-crossing formula for the classical Donaldson-Thomas invariants of the moduli spaces of ADHM sheaves \cite{chuang_chamber,diaconescu_chamber}.

Motivic ADHM recursion formula allows us to express the Higgs sheaf invariants (Section \ref{sec:Higgs}) in terms of the asymptotic ADHM invariants (Section \ref{sec:Asymp}). The later invariants were determined in \cite{chuang_motivic} by string theoretic techniques. This formula is still conjectural from the mathematical point of view. We give its motivic generalization in Conjecture \ref{conj:A}.

Finally, let us discuss those situations where Conjecture \ref{E twisted} on the motives of the moduli spaces of twisted Higgs bundles is known to be true. In the case of untwisted Higgs bundles it is a motivic version of the conjecture of Hausel and Rodriguez-Villegas which was verified in \cite{hausel_mirrora,hausel_mixed}
for rank $2$ using the results of Hitchin \cite{hitchin_self-duality} and for rank $3$ and small genus using the results of Gothen \cite{gothen_betti}. Recently it was checked in \cite{garcia-prada_motive} also for rank $4$ and small genus. In the twisted case our conjecture was checked by Rayan \cite{rayan_geometry} for genus $0$, rank $\le5$, and twisting parameter $\le4$.

I would like to thank Tamas Hausel for drawing my attention to the paper \cite{chuang_motivic} and for many useful discussions.
I would like to thank Duiliu-Emanuel Diaconescu, Jochen Heinloth, and Steven Rayan for many helpful remarks. The author's research was supported by the EPSRC grant EP/G027110/1.

\section{Preliminaries}
\subsection{Partitions}
For any partition $\la=(\la_1,\la_2,\dots)$ (see \eg \cite{Macdonald95}) we define its diagram to be the set
$$d(\la)=\sets{(i,j)}{i\ge1,1\le j\le\la_i}.$$
The elements of the diagram \la are called boxes.
For any box $x=(i,j)\in d(\la)$, we define 
the arm, leg, and hook lengths of $x$ by
$$a(x)=\la_i-j,\quad l(x)=\la'_j-i,\quad h(x)=a(x)+l(x)+1,$$
where $\la'$ denotes the conjugate partition.
Define \cite[1.5]{Macdonald95}
$$n(\la)=\sum_{i\ge1}(i-1)\la_i=\sum_{j\ge1}\binom{\la'_j}2=\sum_{x\in d(\la)}l(x).$$
One can easily show
$$\sum_{x\in d(\la)}(i(x)-l(x)-1)=0.$$

\subsection{\texorpdfstring{\la}{lambda}-Rings}
For the definition and basic properties of \la-rings see \eg \cite{getzler_mixed,mozgovoy_computational}. Given a complete filtered \la-ring which is an algebra over $\cQ$, we define plethystic operators \cite{mozgovoy_computational}
$$\Psi:=\sum_{k\ge1}\frac1k\psi_k,\quad \Psi\inv:=\sum_{k\ge1}\frac{\mu(k)}k\psi_k.$$
We define plethystic exponential and plethystic logarithm
$$\Exp:=\exp\circ\Psi,\quad \Log:=\Psi\inv\circ\log.$$
Given a \la-ring $R$ which is an algebra over \cQ, we will endow the algebra
$R\pser {x_1,\dots,x_n}$ with a \la-ring structure by extending the Adams operators
$$\psi_k(ax^\al)=\psi_k(a)x^{k\al},\qquad a\in R,\al\in\cN^{n}.$$

\subsection{Ring of motives}
Let $\lM=K_0(CM_\cC)$ be the Grothendieck ring of effective Chow motives over \cC with rational coefficients. It is known that $\lM$ is a (special) \la-ring \cite{getzler_mixed,heinloth_note}.
Let $\cL=[\cA^1]\in\lM$. Let $\what\lM$ be the dimensional completion of $\lM[\cL\inv]$ (\cf \cite{behrend_motivica,mozgovoy_wall-crossing}). Finally, let $\lV=\what\lM[\cL^\oh]$. 
To shorten the notation, we will often denote $\cL^\oh$ by $y$.
The ring \lV still has a structure of a \la-ring, where we extend the Adams operations by 
$$\psi_n(\cL^\oh)=\cL^{\frac n2}.$$
Note that the elements $1-\cL^n$, $[\GL_n]$ are invertible in $\what\lM$ and \lV.

For any algebraic variety $X$, we define its motivic zeta-function
$$Z_X(t)=\sum_{n\ge0}[S^nX]t^n=\Exp([X]t)\in \lV\pser t.$$
It is known \cite{heinloth_note,kapranov_elliptic, larsen_rationality} that for a curve $X$ of genus $g$
\begin{equation}
Z_X(t)=\frac{P_X(t)}{(1-t\cL)(1-t)},
\label{eq:P_X}
\end{equation}
where $P_X(t)$ is a polynomial of degree $2g$.
Moreover,
\begin{equation}
Z_X(1/t\cL)=(t^2\cL)^{1-g}Z_X(t).
\label{eq:Z symmetry}
\end{equation}

\subsection{
\texorpdfstring{$E$}{E}-polynomials}
For any complex algebraic variety $X$, there is a mixed Hodge structure on the ordinary cohomology groups $H^*(X,\cQ)$ and on the compactly supported cohomology groups $H^*_c(X,\cQ)$, see \cite{deligne_theoriea}.
We define
$$h^{p,q,k}(X)=\dim \Gr_p^F\Gr_{p+q}^W H^k(X,\cC),\qquad h_c^{p,q,k}(X)=\dim \Gr_p^F\Gr_{p+q}^W H^k_c(X,\cC).$$
Define the mixed Hodge polynomial of $X$ by the formula
$$H(X,u,v,t)=\sum_{p,q,k}h^{p,q,k}(X)u^pv^qt^k.$$
Define the $E$-polynomial of $X$ \cite{danilov_newton} by the formula
$$E(X,u,v)=\sum_{p,q,k}(-1)^k h_c^{p,q,k}(X)u^pv^q.$$
For example, for a curve $X$ of genus $g$, we have
$$E(X,u,v)=1-g(u+v)+uv.$$
The map $X\mto E(X)$ extends to a \la-ring homomorphism
$$E:\lV\to \cQ[u,v]\pser{(uv)\inv}[(uv)^\oh],$$
where the \la-ring structure on the right is given by
$$\psi_n(f(u,v))=f(u^n,v^n).$$
Applying the $E$-polynomial map to the motivic zeta-function $Z_X(t)$ we obtain
$$Z_X^H(t,u,v)=\sum_{n\ge0}E(S^nX,u,v)t^n=\Exp(E(X,u,v)t).$$
If $X$ is a curve of genus $g$ then
\begin{equation}
\label{eq:A}
Z_X^H(t,u,v)=\Exp(E(X,u,v)t)=\frac{(1-tu)^g(1-tv)^g}{(1-tuv)(1-t)}.
\end{equation}

\begin{rmr}
\label{y-genus1}
Note that the $\hi_y$-genus specialization
$$Z_X^H(t,y,1)=\frac{(1-ty)^g(1-t)^g}{(1-ty)(1-t)}$$
is a polynomial in $t,y$ if $g\ge1$.
\end{rmr}

\section{Invariants of Higgs moduli spaces}
\subsection{Hausel-Rodriguez-Villegas conjecture}
\label{sec:hrv}
Let $X$ be a curve of genus $g$.
Following \cite[Remark 5.5.9]{hausel_mirrora}
we define, for any partition \la,
\rem{I have changed the signs of $u,v$ to be able to use the \la-rings formalism later}
$$\lH'_\la(t,u,v)
=(tuv)^{(2-2g)n(\la)}\prod_{x\in d(\la)}
\frac{(1-t^hv^lu^{l+1})^g(1-t^hu^lv^{l+1})^g}{(1-t^h(uv)^{l+1})(1-t^h(uv)^l)}.$$

\begin{rmr}
\label{rmr:zeta}
We have
\begin{equation}
\lH'_\la(t,u,v)=(tuv)^{(2-2g)n(\la)}
\prod_{x\in d(\la)}Z_X^H(t^h(uv)^l,u,v),
\label{eq:lH}
\end{equation}
where $Z_X^H(t,u,v)$ is given by equation \eqref{eq:A}.
\end{rmr}
\begin{rmr}
One defines in \cite[2.4.10]{hausel_mixed} a slightly different function
$$\lH''_\la(z,w)=\prod_{x\in d(\la)}\frac{(z^{2a+1}-w^{2l+1})^{2g}}
{(z^{2a+2}-w^{2l})(z^{2a}-w^{2l+2})}.$$
It is related to $\lH'_\la(t,u,v)$ by
$$\lH''_\la(t^{1/2},1/yt^{1/2})
=(ty^2)^{\n\la(1-g)}\lH'_\la(t,y,y).$$
\end{rmr}

Define the rational functions $H'_n\in\cQ(t,u,v)$ by the formula
\begin{equation}
\sum\lH'_\la(t,u,v)T^{\n\la}=\Exp
\lrb{\sum_{n\ge1}\frac{(tuv)^{(1-g)(n^2-n)}}{(1-tuv)(1-t)}H'_n(t,u,v)T^n}.
\label{eq:H_n}
\end{equation}

It is conjectured in \cite[4.2.3]{hausel_mixed} that
$H'_n(t,u,v)$ are polynomials and $H'_n(t,-u,-v)$ have nonnegative coefficients.
Moreover, the Hodge polynomial for the ordinary cohomologies of the moduli space $\lM(n,d)$ of stable Higgs bundles with coprime rank $n$ and degree $d$ on a curve $X$ is given conjecturally \cite[Conj.~5.6]{hausel_mirrora} by 
\begin{equation}
H(\lM(n,d),u,v,-1)=H'_n(1,u,v).
\end{equation}

\begin{rmr}
The moduli space $\lM(n,d)$ of stable Higgs bundles is smooth and has dimension $2((g-1)n^2+1)$. The Hodge structure on $H^*(\lM(n,d),\cQ)$ is pure \cite[Theorem 2.1]{hausel_mirrora}.
\end{rmr}

\begin{rmr}
It follows from Remark \ref{y-genus1} that the $y$-genus specialization $H'(t,y,1)$ is a polynomial in $t^{\pm1},y$ if $g\ge1$.
\end{rmr}

\begin{rmr}
It is conjectured in \cite[Conj.~4.2.1]{hausel_mixed},
\cite[Conj.~5.1]{hausel_mirrora} that the mixed Hodge
polynomial of the character variety $\lM_B(\GL_n(\cC))$ is
given by
$$H(\lM_B(\GL_n(\cC)),u,v,t)=H'_n(uv,-t,-t).$$
This is proved in
\cite{hausel_mixed} for $t=1$.
\end{rmr}

\subsection{
\texorpdfstring{$E$}{E}-polynomials of Higgs moduli spaces}
\label{sec:E-poly}
The conjectural polynomial\br
 $H'_n(1,u,v)$ describes the Hodge polynomial for ordinary cohomologies of the moduli space $\lM(n,d)$ of Higgs bundles. In order to get the motivic invariants of $\lM(n,d)$, we have to pass to the $E$-polynomial (\ie to the Hodge polynomial for cohomologies with compact support). This is done by applying the Poincar\'e duality to the smooth variety $\lM(n,d)$. The $E$-polynomial of $\lM=\lM(n,d)$ should be equal to 
$$(uv)^{\dim\lM}H'_n(1,u\inv,v\inv).$$
Let 
$$H_n(t,u,v)=H'_n(tuv,u\inv,v\inv).$$
Using equations \eqref{eq:lH} and \eqref{eq:H_n} we can show that $H_n(t,u,v)$ satisfy the following equation
\begin{multline*}
\sum_{\la}t^{(1-g)2n(\la)}\left(\prod_{x\in d(\la)}Z_X^H(t^h(uv)^a,u,v)\right)T^{\n\la}\\
=\Exp\left(\sum_{n\ge1}\frac{t^{(1-g)(n^2-n)}}{(1-t)(1-tuv)}H_n(t,u,v)T^n\right).
\end{multline*}
As we will see later (see Lemma \ref{palyndr2}), this equation implies that
$$t^{-\dim\lM/2}H_n(t,u,v)$$
is invariant under the change of variables $t\mto1/tuv$. In particular
$$H_n(1,u,v)=(uv)^{\dim\lM/2}H_n(1/uv,u,v)$$
and therefore the $E$-polynomial of $\lM$ should be equal to
$$(uv)^{\dim\lM}H'_n(1,u\inv,v\inv)
=(uv)^{\dim\lM}H_n(1/uv,u,v)
=(uv)^{\dim\lM/2}H_n(1,u,v).$$ 

We can reformulate now the conjecture of Hausel and Rodriguez-Villegas in terms of $E$-polynomials
\begin{conj}
For any partition \la define
\begin{multline*}
\lH_\la(t,u,v)=t^{(1-g)(2n(\la)+\n\la)}\prod_{x\in d(\la)}Z_X^H(t^h(uv)^a,u,v)\\
=\prod_{x\in d(\la)}t^{(1-g)(2l+1)}Z_X^H(t^h(uv)^a,u,v).
\end{multline*}
Define the functions $H_n(t,u,v)\in\cZ[u,v]\lser t$ by
$$\sum_{\la}\lH_\la(t,u,v)T^{\n\la}
=\Exp\left(\sum_{n\ge1}\frac{t^{(1-g)n^2}H_n(t,u,v)}{(1-t)(1-tuv)}T^n\right).$$
Then $H_n(t,u,v)$ are polynomials and we have
$$E(\lM,u,v)=(uv)^{\dim\lM/2}H_n(1,u,v),$$
where $\lM=\lM(n,d)$ is the moduli space of Higgs bundles having coprime rank $n$ and degree $d$, and $\dim\lM=2((g-1)n^2+1)$.
\end{conj}

\subsection{Motives of Higgs moduli spaces}
In the previous section we have used the function $Z^H_X(t,u,v)$ to formulate the conjecture. This function is an $E$-polynomial of the motivic zeta-function $Z_X(t)$ of the curve $X$. We can use this motivic zeta-function to formulate

\begin{conj}
\label{E untwisted}
For any partition $\la$ define
$$\lH_\la(t)=t^{(1-g)(2n(\la)+\n\la)}\prod_{x\in d(\la)}Z_X(t^h\cL^a)
=\prod_{x\in d(\la)}t^{(1-g)(2l+1)}Z_X(t^h\cL^a)\in\lV\lser t.$$
Define the functions $H_n(t)\in\lV\lser t$ by
$$\sum_{\la}\lH_\la(t)T^{\n\la}=\Exp\left(\sum_{n\ge1}\frac{t^{(1-g)n^2}H_n(t)}{(1-t)(1-t\cL)}T^n\right).$$
Then $H_n(t)$ are polynomials in $t$ and we have
$$[\lM]=\cL^{\dim\lM/2}H_n(1),$$
where $\lM=\lM(n,d)$ is the moduli space of Higgs bundles having coprime rank $n$ and degree $d$, and $\dim\lM=2((g-1)n^2+1)$.
\end{conj}

\subsection{Twisted Higgs bundles}
Let $L$ be a line bundle on $X$ of degree $2g-2+p$, where $p\ge0$. An $L$-twisted Higgs bundle is a vector bundle $E$ with a morphism $\vi:E\to E\ts L$. Let $\lM(L,n,d)$ denote the moduli space of semistable $L$-twisted Higgs bundles of rank $n$ and degree $d$.
We are going to describe a conjectural formula for its motive in the case of coprime rank and degree. This formula can be specialized to the $E$-polynomial, Poincar\'e polynomial, $y$-genus, and Euler number of $\lM(L,n,d)$ in the obvious way.

\begin{conj}
\label{E twisted}
For any partition \la define
\begin{multline}
\label{eq:Hp la}
\lH_\la^{(p)}(t)=(-1)^{p\n\la}t^{(1-g)(2n(\la)+\n\la)+p(n(\la')-n(\la))}\cL^{pn(\la')}\prod_{x\in d(\la)}Z_X(t^h\cL^a)\\
=\prod_{x\in d(\la)}(-t^{a-l}\cL^a)^p t^{(1-g)(2l+1)}Z_X(t^h\cL^a)\in\lV\lser t.
\end{multline}
Define the functions $H_n^{(p)}(t)\in\lV\lser t$ by
\begin{equation}
\sum_\la\lH_\la^{(p)}(t)T^{\n\la}=\Exp\left(\sum_{n\ge1}\frac{(-1)^{pn}t^{(1-g)n^2-p\binom n2}H_n^{(p)}(t)}{(1-t)(1-t\cL)}T^n\right).
\label{eq:Hp}
\end{equation}
Then $H_n^{(p)}(t)$ are polynomials in $t$ and we have
$$[\lM]=\cL^{\dim\lM/2}H_n^{(p)}(1),$$
where $\lM=\lM(L,n,d)$ is the moduli space of twisted Higgs bundles having coprime rank $n$ and degree $d$, and $\dim\lM=2((g-1)n^2+p\binom n2+1)$.
\end{conj}

\begin{rmr}
\label{symmetric H}
Computer tests show that the degree of $H_n^{(p)}(t)$ equals $\dim\lM$.
Let us define
$$\wtl H_n^{(p)}(t)=(-1)^{pn}t^{-\dim\lM/2}H_n^{(p)}(t).$$
Then
$$\sum_\la\lH_\la^{(p)}(t)T^{\n\la}=\Exp\left(\sum_{n\ge1}\frac{t}{(1-t)(1-t\cL)}\wtl H_n^{(p)}(t)T^n\right).$$
We will see in Lemma \ref{palyndr2} that
$$\wtl H_n^{(p)}(1/t\cL)=\wtl H_n^{(p)}(t).$$
\end{rmr}

\begin{rmr}
Let $H^{(p)}_n(t,u,v)\in\cQ(u,v)\lser t$ be the $E$-polynomial specialization of $H_n^{(p)}\in\lV\lser t$. Our tests show that $H^{(p)}_n(t,u,v)$ is a polynomial with integer coefficients and $H^{(p)}_n(t,-u,-v)$ has non-negative coefficients. \end{rmr}

\begin{rmr}
Computer tests show that the series
$$\frac{1}{P_X(t)}H^{(p)}_n(t)$$
should be a polynomial in $t$, where the polynomial $P_X(t)$ was defined in \eqref{eq:P_X}.
\end{rmr}

\begin{rmr}
\label{Hodge twisted}
We can reformulate the last conjecture to get the Hodge polynomials for the twisted Higgs moduli spaces. The Hodge polynomial of $\lM=\lM(L,n,d)$, for coprime $n,d$, should be equal to $H_n'^{(p)}(1,-u,-v)$, where $H_n'^{(p)}(t,u,v)$ are given by (the formula is obtained by the substitution $H_n^{(p)}(t,u,v)=H_n'^{(p)}(tuv,u\inv,v\inv)$ used earlier in Section \ref{sec:E-poly})
\begin{multline*}
\sum_\la\left(\prod_{x\in d(\la)}(-t^{a-l}(uv)^{-l})^p (tuv)^{(1-g)(2l+1)}Z_X^H(t^h(uv)^l)\right)T^{\n\la}\\
=\Exp\left(\sum_{n\ge1}\frac{(-1)^{pn}(tuv)^{(1-g)n^2-p\binom n2}H_n'^{(p)}(t,u,v)}{(1-t)(1-tuv)}T^n\right).
\end{multline*}
\end{rmr}

\section{Recursion formula}
There are three ingredients in the ADHM recursion formula: recursion formula itself, asymptotic ADHM invariants, and Higgs sheaf invariants. We start with the second and third ingredients.

\subsection{ADHM sheaves}
We define an ADHM quiver $Q$ to be the quiver with two vertices $1,*$ and four arrows
$$\Phi_1,\Phi_2:1\to1,\quad \vi:1\to*,\quad \psi:*\to1.$$
Let $M=(M_1,M_2,M_\vi,M_\psi)$ be a tuple of line bundles on a curve $X$ associated to the arrows of $Q$. An $M$-twisted $Q$-sheaf on $X$ (\cf \cite{gothen_homological}) is a pair of sheaves $(E,E_*)$ associated to the vertices of $Q$ together with morphisms
$$\Phi_1:E\ts M_1\to E,\quad \Phi_2:E\ts M_2\to E,\quad \phi:E\ts M_\phi\to E_*,\quad \psi:E_*\ts M_\psi\to E.$$
Assume that
$$M_\phi=M_1\ts M_2,\quad M_\psi=O_X.$$
We say that an $M$-twisted $Q$-sheaf $\Eps=(E,E_*,\Phi_1,\Phi_2,\phi,\psi)$ satisfies the ADHM relation if
$$\Phi_1(\Phi_2\ts 1_{M_1})-\Phi_2(\Phi_1\ts M_2)+\psi\phi=0.$$
An ADHM sheaf is an $M$-twisted $Q$-sheaf $\Eps=(E,E_*,\Phi_1,\Phi_2,\phi,\psi)$ satisfying the ADHM relation such that $E_*\iso V\ts O_X$ for some vector space $V$.
We will write an ADHM sheaf as a tuple $\Eps=(E,V,\Phi_1,\Phi_2,\phi,\psi)$. We define
$$r(\Eps)=\rk E+\dim V,\quad \hi(\Eps)=\hi(E),\quad v(\Eps)=\dim V.$$
For any real number $c\in\cR$, we define the slope function
$$\mu_c(\Eps)=\frac{\hi(\Eps)+cv(\Eps)}{r(\Eps)}$$
and then define the stability condition with respect to this slope function.
Let $\lM_c(M,r,d,v)$ be the moduli space of semistable ADHM sheaves \lE such that
$$r(\Eps)=r+v,\qquad \hi(\Eps)=d+(1-g)r,\qquad v(\Eps)=v.$$
This moduli space does not change for $c\gg0$
and we will denote it by \br
$\lM_{+\infty}(M,r,d,v)$. If $v=0$ then $\lM_c(M,r,d,v)$ is independent of $c$ and will be denoted by $\lM(M,r,d,0)$

Given a line bundle $L$ of degree $2g-2+p$, where $p\ge0$, we define 
$$M_1=K_X\inv \ts L,\quad M_2=L\inv,\quad M_\phi=M_1\ts M_2\iso K_X\inv,\quad M_\psi=O_X.$$
The corresponding moduli space will be denoted by $\lM_c(L,r,d,v)$.
We will be only interested in the ADHM sheaves $\lE$ with $v(\Eps)=0,1$.
It is proved in \cite{chuang_motivic} that 
$$\lM(K_X,r,d,0)\iso\cC\xx\lM(r,d),\quad \lM(L,r,d,0)\iso\lM(L,r,d)$$
if $\deg L>2g-2$ and $r,d$ are coprime.

\subsection{Asymptotic ADHM invariants}
\label{sec:Asymp}
Let $L$ be a line bundle of degree $2g-2+p$, where $p\ge0$.
The asymptotic motivic ADHM invariants 
$$A_{+\infty,\ga}\in\lV,\qquad \ga=(r,d)\in\cZ_{\ge0}\xx\cZ,$$
are the (conjectural) motivic Donaldson-Thomas invariants of the moduli spaces
$\lM_{+\infty}(L,r,d,1)$ of ADHM sheaves.
The conjectural formula for these invariants in the context of Hodge polynomials for ordinary cohomologies was given in \cite[Eq.~1.12]{chuang_motivic}. These polynomials are not motivic invariants, so we should pass to the $E$-polynomials.
We will formulate actually a motivic version of this conjecture. In what follows we will use a new variable $s=t\cL^\oh$. Note that the change of variables $t\mto1/t\cL$ corresponds to $s\mto s\inv$.

\begin{conj}
\label{conj:A}
We have
\rem{We have $s/y=1/ty^2=1/t\cL$}
$$A_{+\infty}:=\sum_{\ga}A_{+\infty,\ga}s^\hi \TT^r
=\sum_\la\lH^{(p)}_\la(t)\TT^{\n\la}\in\lV\lser{s}\pser T,$$
where the function $\lH_\la^{(p)}$ was defined by equation \eqref{eq:Hp la} and, for $\ga=(r,d)$, we define $\hi=\hi(\ga):=d+(1-g)r$.
\end{conj}

\begin{rmr}
Let us discuss the relation of this conjecture to the conjecture given in \cite[Eq.~1.12]{chuang_motivic}. For simplicity we will consider only the case $p=0$. The invariants $A_{+\infty,\ga}'(u,v)$ defined in \cite[Eq.~1.12]{chuang_motivic} are based on the Hodge polynomials for ordinary cohomologies. They should satisfy
\begin{equation*}
\sum_{\ga=(r,d)}A'_{+\infty,\ga}(u,v)s^d\TT^r
=\sum_\la \Om_\la^{(0)}(s,u,v)\TT^{\n\la},
\end{equation*}
where the left sum runs over $r\ge0$, $d\in\cZ$, the right sum runs over all partitions \la, and $\Om_\la^{(0)}\in\cQ(s,u^\oh,v^\oh)$ are given by \cite[Eq.~1.13]{chuang_motivic} (we use $y=(uv)^\oh$, $s=ty$ as usual)
\begin{multline*}
\Om^{(0)}_\la(s,u,v)\\
=\prod_{s\in\la}(uv)^{(g-1)(-2a+l-i+1)/2} s^{(g-1)(-2a-l+i-1)} (uv)^{(1-g)/2}\\
\xx\frac{(1-s^hu^{(a-l+1)/2}v^{(a-l-1)/2})^g(1-s^hv^{(a-l+1)/2}u^{(a-l-1)/2})^g}
{(1-s^h(uv)^{(a-l+1)/2})(1-s^h(uv)^{(a-l-1)/2})}\\
=y^{(1-g)\n\la}\prod (sy)^{(1-g)2a}
\frac{(1-(s/y)^{h}u^{a+1}v^{a})^g(1-(s/y)^{h}v^{a+1}u^{a})^g}
{(1-(s/y)^{h}(uv)^{a+1})(1-(s/y)^{h}(uv)^{a})}\\
=y^{(1-g)\n\la}\prod (tuv)^{(1-g)2a}Z_X^H(t^h(uv)^a,u,v)
=y^{(1-g)\n\la}\lH'_{\la'}(t,u,v).
\end{multline*}
This implies
\begin{multline*}
\sum_{\ga=(r,d)}A'_{+\infty,\ga}(u,v)s^\hi\TT^r
=\sum_\la \Om_\la^{(0)}(s,u,v)(s^{1-g}\TT)^{\n\la}\\
=\sum_\la\lH'_\la(t,u,v)((sy)^{1-g}\TT)^{\n\la}
=\sum_\la\lH'_\la(t,u,v)((tuv)^{1-g}\TT)^{\n\la}.
\end{multline*}
After passing from the Hodge polynomials for ordinary cohomologies to the $E$-polynomials, we get our formula. 
\end{rmr}

\subsection{Higgs sheaf invariants}
\label{sec:Higgs}
The motivic Higgs sheaf invariants
$$\Om_\ga\in\lV,\qquad \ga=(r,d)\in\cZ_{\ge0}\xx\cZ,$$
are the (conjectural) motivic Donaldson-Thomas invariants of the moduli spaces $\lM(L,r,d,0)$ of Higgs sheaves. For $\ga=(r,d)$ we define $\hi=\hi(\ga)=d+(1-g)r$ and define the slope $\mu(\ga)=\hi/r$.
For any $\mu\in\cR$ let 
$$B_\mu=\sum_{\mu(\ga)=\mu}B_\ga s^\hi T^r$$
be the motivic Donaldson-Thomas series \cite{kontsevich_stability} of the moduli spaces of Higgs sheaves having slope $\mu$. Then the motivic Donaldson-Thomas invariants $$\Om_\mu=\sum_{\mu(\ga)=\mu}\Om_\ga s^\hi T^r$$
are defined by the formula (\cf \cite{kontsevich_cohomological,mozgovoy_wall-crossing})
$$B_\mu=\Exp\left(\frac{\Om_\mu}{\cL-1}\right).$$

The conjectural formula for these invariants in the context of Hodge polynomials for ordinary cohomologies was given in
\cite[Conj.~1.5]{chuang_motivic}. We formulate a motivic version of this conjecture.

\begin{conj}
\label{conj:Om}
Invariants $\Om_{\ga}$, $\ga=(r,d)$, are independent of degree
and are given by the formula
\begin{equation}
\Om_\ga=\Om_r=(-1)^{pr}\cL^{-\dim\lM/2}[\lM]=(-1)^{pr}H_r^{(p)}(1),
\label{eq:ubH}
\end{equation}
where $\lM=\lM(L,r,d')$ (for some $d'$ coprime to $r$) is the moduli space of Higgs bundles and the function $H_r^{(p)}$ was defined by equation \eqref{eq:Hp}.
\end{conj}

\subsection{Recursion formula}
The recursion formula \cite[Eq.~1.6]{chuang_motivic} is a wall-crossing formula that
relates the asymptotic ADHM invariants and the Higgs sheaf invariants.

\begin{dfn}
Let 
$$f=\sum_\ga f_\ga s^\hi \TT^r\in\lV\pser{s^{\pm1},\TT}.$$
For any $\mu\in\cR$, we define
$$f_{\mu}=\sum_{\mu(\ga)=\mu}f_\ga s^\hi \TT^r.$$
In the same way we define $f_{>\mu}$, $f_{\ge\mu}$, $f_{<\mu}$, and $f_{\le\mu}$.
\end{dfn}

The following conjecture is a version of the ADHM recursion formula \cite[Eq.~1.6]{chuang_motivic} \begin{conj}
\label{conj:rec1}
For any $\mu\in\cR$, we have
\begin{equation}
\left(A_{+\infty}C_{>\mu}\inv\right)_{\mu}
=\left(A_{+\infty}C_{\ge-\mu}\inv\right)_{-\mu}(s\inv),
\label{eq:recursion2}
\end{equation}
where
$$C_{>\mu}=\prod_{\eta>\mu}C_{\eta},\qquad
C_{\ge\mu}=\prod_{\eta\ge\mu}C_{\eta},$$
\begin{equation}
C_\mu=S_{\hi}B_\mu\cdot S_{-\hi}B_\mu\inv
=\Exp\left((S_{\hi}-S_{-\hi})\frac{\Om_\mu}{\cL-1}\right),
\label{eq:B'}
\end{equation}
\rem{$C=(B')\inv$}
and for any group homomorphism $\la:\cZ^2\to\cZ$ we define the operator $$S_\la:\lV\pser{s^{\pm1},\TT}\to\lV\pser{s^{\pm1},\TT},\qquad s^k\TT^r\mto\cL^{\oh\la(r,k)}s^k\TT^r.$$
\end{conj}

\begin{rmr}
The wall-crossing formula described in the above conjecture is very similar to \cite[Corollary 4.25]{mozgovoy_wall-crossing}. The difference is that our $A_{+\infty}$ equals to $A_{+\infty}(-\cL^\oh)$ from \cite{mozgovoy_wall-crossing} and the operator $S_\la$ in \cite{mozgovoy_wall-crossing} is given by $s^k\TT^r\mto(-\cL^\oh)^{\la(r,k)}s^k\TT^r$. Our operator $S_\la$ has a useful property that it commutes with \Exp.
\end{rmr}

\begin{rmr}
\label{rmr:recursion}
Later we will see that $C_{-\mu}\inv(s\inv)=C_\mu$
(see Remark \ref{rmr:mu and -mu1}).
Therefore the formula \eqref{eq:recursion2} can be written in the form
\begin{equation}
C_\mu=\left(\ub A_{+\infty}C_{>\mu}\inv\right)_{\mu}
-\left(\ub A_{+\infty}C_{\ge-\mu}\inv\right)_{-\mu}(s\inv),
\label{eq:recursion conjecture}
\end{equation}
where $\ub A_{+\infty}=A_{+\infty}-1$.
This formula can be used to determine the invariants $\Om_\ga$ inductively by rank because the summands of $\ub A_{+\infty}$ have only positive powers of $\TT$. For this reason equation \eqref{eq:recursion2} is called a recursion formula.
\end{rmr}

\begin{rmr}
The original recursion formula \cite[Eq.~1.6]{chuang_motivic} is actually somewhat different from the formulated conjecture. Let me explain how they are related.
One defines (we use $y=\cL^\oh$ as usual)
$$\ub H_{(r,d)}=\ub H_r:=y\inv\Om_r$$
for arbitrary $d$ and defines invariants $H_\ga$ by the multicover formula \cite[Conj.~1.2]{chuang_motivic} 
$$
H_\ga=\sum_{k\mid\ga}\frac{1}{k[k]_y}\psi_k(\ub H_{\ga/k}),\quad [k]_y=\frac{y^k-{y^{-k}}}{y-y\inv}.
$$
\rem{Chuang etal use $\tl H_\ga$ instead of our $H_\ga$ and their $H_\ga$ equals $(-1)^\hi \tl H_\ga$, so our recursion formula is identical to theirs, but we have changed the meaning of $H_\ga$}
The recursion formula \cite[Eq.~1.6]{chuang_motivic} can be written in the form
\begin{equation}
\exp(\gC_\mu)=(\ub A_{+\infty}\exp(-\gC_{>\mu}))_{\mu}
-(\ub A_{+\infty}\exp(-\gC_{\ge-\mu}))_{-\mu}(s\inv),
\label{eq:rec orig}
\end{equation}
where
$$\gC=\sum_{\ga}[\hi]_yH_\ga s^\hi \TT^r \in \lV\pser{s^{\pm1},\TT}.$$
Note that
\begin{multline*}
\sum_{\mu(\ga)\ge\mu}\frac{H_\ga}{y-y\inv} s^\hi\TT^r
=\sum_{k\ge1}\sum_{\mu(\ga)\ge\mu} \frac1{k[k]_y(y-y\inv)}\psi_k(\ub H_\ga) s^{k\hi}\TT^{kr}\\
=\sum_{k\ge1}\frac1k \psi_k\left(\frac{\sum_{\mu(\ga)\ge\mu}\ub H_\ga s^\hi\TT^r}{y-y\inv}\right)
=\Psi\left(\frac{\sum_{\mu(\ga)\ge\mu}y\ub H_\ga s^\hi\TT^r}{\cL-1}\right)
=\Psi\left(\frac{\Om_{\ge\mu}}{\cL-1}\right)
\label{eq:generating H_{r,d}}
\end{multline*}
and therefore
$$\gC_{\ge\mu}
=(S_{\hi}-S_{-\hi})\sum_{\mu(\ga)\ge\mu}\frac{H_\ga}{y-y\inv}s^\hi\TT^r\\
=(S_{\hi}-S_{-\hi})\Psi\left(\frac{\Om_{\ge\mu}}{\cL-1}\right).$$
This implies
$$
\exp(\gC_{\ge\mu})
=\exp\left((S_{\hi}-S_{-\hi})\Psi\left(\frac{\Om_{\ge\mu}}{\cL-1}\right)\right)
=\Exp\left((S_{\hi}-S_{-\hi})\frac{\Om_{\ge\mu}}{\cL-1}\right)
=C_{\ge\mu}	
$$
and one can see that formulas \eqref{eq:recursion conjecture} and \eqref{eq:rec orig} are equivalent.
\end{rmr}

\subsection{Solutions}
The main result of this paper is the following

\begin{thr}
\label{main1}
Assume that the functions $H_n^{(p)}(t)$ defined in \eqref{eq:Hp} are polynomials in~$t$. If Conjecture \ref{conj:A} (description of asymptotic ADHM invariants) and Conjecture \ref{conj:rec1} (ADHM recursion formula) are true, then the Higgs sheaf invariants are described by Conjecture \ref{conj:Om}.
\end{thr}

This result was checked for rank $r\le3$ and small $g$ in \cite{chuang_motivic} by explicit calculations. It follows from Remark \ref{rmr:recursion} that the values $\Om_\ga$ are uniquely determined by the recursion formula and the series $A_{+\infty}$. Therefore to prove Theorem \ref{main1} we have to show that invariants given by Conjectures \ref{conj:A}, \ref{conj:Om} satisfy the recursion formula. We will prove a more general result

\begin{thr}
\label{main2}
For any $\mu\in\cR$ we have
\begin{equation}
A_{+\infty}C_{>\mu}\inv
=(A_{+\infty}C_{\ge-\mu}\inv)(s\inv)
\label{eq:Th}
\end{equation}
in $\lV\pser{s^{\pm1},\TT}$, where on the right side we first
embed $A_{+\infty}C_{\ge-\mu}\inv$ into 
$\lV\lser{s}\pser{\TT}$ and then invert $s$.
\end{thr}


\section{Properties of the ADHM invariants}
In this section we will prove Theorem \ref{main2}.
We will see that $C_{>\mu}$ and $C_{\ge-\mu}$
are contained in $\lV(s)\pser \TT$ (we say that they are $\TT$-rational in this case)
and satisfy
$$C_{>\mu}(s)=C_{\ge-\mu}(s\inv)$$
in $\lV(s)\pser \TT$.
Also $A_{+\infty}$ is $\TT$-rational and we have
$A_{+\infty}(s)=A_{+\infty}(s\inv)$ in $\lV(s)\pser \TT$.
These facts imply that an analog of equation \eqref{eq:Th} holds in $\lV(s)\pser \TT$.
To prove Theorem \ref{main2} we will show that $A_{+\infty}C_{>\mu}\inv$ is contained in $\lV[s^{\pm1}]\pser \TT$.

We use the same notation as in the previous sections. We will denote $\lH^{(p)}_\la$ (resp.\ $H^{(p)}_n$) just by $\lH_\la$ (resp.\ $H_n$).

\subsection{Properties of Higgs sheaf invariants}
Recall from equation \eqref{eq:B'} that
$$C_\mu=\Exp\left((S_{\hi}-S_{-\hi})\frac{\Om_\mu}{\cL-1}\right).$$
Define 
\begin{equation}
\Om'_\mu=(S_{\hi}-S_{-\hi})\Om_\mu
\label{eq:Om'mu}
\end{equation}

and more generally
\begin{multline}
\Om'=(S_{\hi}-S_{-\hi})\sum_\ga\Om_\ga s^\hi T^r
=\sum_{n\ge1}\Om_n T^n\sum_{k\in\cZ}((sy)^k-(s/y)^k)\\
=\Om(T)\sum_{k\in\cZ}((sy)^k-(s/y)^k)\in\lV\pser{s^{\pm1},T},
\end{multline}
where $\Om(T):=\sum_{n\ge1}\Om_n T^n$.

\begin{lmm}
\label{mu and -mu}
We have $\OO_{>\mu}(s)=-\OO_{<-\mu}(s\inv)$ in $\lV\pser{s^{\pm1},\TT}$.
\end{lmm}
\begin{proof}
If $s^k\TT^r$ appears in $\OO_{>\mu}$ then $k/r>\mu$. Therefore $-k/r<-\mu$ and $s^{-k}\TT^r$ appears in $\OO_{<-\mu}$.
The coefficient of $s^k\TT^r$ in $\OO_{>\mu}$ equals
$$(y^{k}-y^{-k})\Om_r,$$
whereas the coefficient of $s^{-k}\TT^r$ in $\OO_{<-\mu}$ equals
to its opposite
$$(y^{-k}-y^{k})\Om_r.$$
\end{proof}

\begin{rmr}
\label{rmr:mu and -mu1}
We also have
$\OO_{\ge\mu}(s)=-\OO_{\le-\mu}(s\inv)$ and
$\OO_{\mu}(s)=-\OO_{-\mu}(s\inv)$.
\end{rmr}

\begin{dfn}
We say that the series $f=\sum_{k\ge0}f_k\TT^k\in\lV\pser{s^{\pm1},\TT}$
is $\TT$-Laurent if $f_k\in\lV\lser s$, $k\ge0$.
We say that $f$ is $\TT$-rational if every $f_k\in\lV\lser s$, $k\ge0$, is rational (this means that there exists $g\in\lV[s]$ invertible in $\lV\pser s$ such that $f_kg\in\lV[s]$, we write $f_k\in\lV(s)$ in this case).
\end{dfn}

\begin{lmm}
\label{lmm:Hprb>0 is palindromic}
The series $\OO_{\ge0}=\OO_{>0}$ is $\TT$-rational and satisfies
$\OO_{>0}(s)=\OO_{\ge0}(s\inv)$ in $\lV(s)\pser \TT$.
\end{lmm}
\begin{proof}
We have (we use $s=ty$, as usual)
\begin{multline*}
\OO_{>0}
=\Om(T)\sum_{k\ge0}((sy)^k-(s/y)^k)\\
=\Om(T)\left(\frac1{1-t\cL}-\frac1{1-t}\right)
=\Om(T)\frac{t(\cL-1)}{(1-t)(1-t\cL)}.
\end{multline*}
But it is clear that
$$\frac{t}{(1-t)(1-t\cL)}$$
is invariant under the change of variables $t\mto1/t\cL$ that corresponds to the change of variables $s\mto s\inv$.
\end{proof}


We can generalize the last lemma to an arbitrary slope.

\begin{prp}
For any $\mu\in\cR$, the series $\OO_{\ge\mu}$ and $\OO_{>\mu}$
are $\TT$-rational and satisfy 
$\OO_{>\mu}(s)=\OO_{\ge-\mu}(s\inv)$
in $\lV(s)\pser \TT$.
\end{prp}
\begin{proof}
We have seen
that $\OO_{\ge0}$ is $\TT$-rational.
The coefficients by $\TT^r$
in $\OO_{\ge0}$ and $\OO_{\ge\mu}$ differ by an element
from $\lV[s^{\pm1}]$. This implies that $\OO_{\ge\mu}$
is also $\TT$-rational. The same applies to $\OO_{>\mu}$.
We have in $\lV\pser{s^{\pm},\TT}$
\begin{equation}
\OO_{\ge-\mu}(s\inv)-\OO_{\ge0}(s\inv)
=\OO_{<0}(s\inv)-\OO_{<-\mu}(s\inv)
=\OO_{>\mu}(s)-\OO_{>0}(s),
\label{eq:diff1}
\end{equation}
where the last equality follows from Lemma \ref{mu and -mu}.
The difference $\OO_{\ge-\mu}(s)-\OO_{\ge0}(s)$ is contained in $\lV[s^\pm]\pser \TT$,
so we can invert $s$ in both expressions either in $\lV(s)\pser \TT$ or in $\lV\pser{s^\pm,\TT}$ obtaining the same difference. Therefore equation \eqref{eq:diff1} implies 
$$\OO_{\ge-\mu}(s\inv)-\OO_{\ge0}(s\inv)
=\OO_{>\mu}(s)-\OO_{>0}(s).$$
in $\lV(s)\pser \TT$.
By Lemma \ref{lmm:Hprb>0 is palindromic} we have
$\OO_{\ge0}(s\inv)=\OO_{>0}(s)$ in $\lV(s)\pser \TT$.
This implies $\OO_{\ge-\mu}(s\inv)=\OO_{>\mu}(s)$.
\end{proof}

\subsection{Properties of asymptotic ADHM invariants}

\begin{lmm}
\label{A property}
We have $A_{+\infty}(s,\TT)=A_{+\infty}(s\inv,\TT)$ in $\lV(s)\pser \TT$.
\end{lmm}
\begin{proof}
Recall that
$$A_{+\infty}(s,T)=\sum_\la\lH_\la(t)\TT^{\n\la}.$$
The change of variable $s\mto s\inv$ corresponds to the change of variables $t\mto 1/t\cL$.
Therefore it is enough to show that
$$\lH_{\la}(1/t\cL)=\lH_{\la'}(t).$$
According to equation \eqref{eq:Hp la} we have
$$\lH_\la(t)
=\prod_{x\in d(\la)}(-t^{a-l}\cL^a)^p\prod_{x\in d(\la)}t^{(1-g)(2l+1)}Z_X(t^h\cL^{a}).
$$
Denote the first product by $f_\la(t)$. Then one can easily see that $f_\la(1/t\cL)=f_{\la'}(t)$. Therefore, we can assume that $p=0$. Applying
equation \eqref{eq:Z symmetry}
$$Z_X(1/t\cL)=(t^2\cL)^{1-g}Z_X(t)$$
we obtain
\begin{multline*}
$$\lH_\la(1/t\cL)
=\prod_{x\in d(\la)}(t\cL)^{(g-1)(2l+1)}Z_X(t^{-h}\cL^{-l-1})\\
=\prod_{x\in d(\la)}(t\cL)^{(g-1)(2l+1)}(t^{2h}\cL^{2l+1})^{1-g}Z_X(t^{h}\cL^{l})
=\prod_{x\in d(\la)}t^{(1-g)(2a+1)}Z_X(t^{h}\cL^{l})
=\lH_{\la'}(t).
\end{multline*}
\end{proof}

Recall from Remark \ref{symmetric H} that we have defined
$$\wtl H_n(t)=(-1)^{pn}t^{-\dim\lM/2}H_n(t),$$
where $\lM=\lM(L,n,d)$ is the moduli space of twisted Higgs bundles. It follows from equation \eqref{eq:ubH} that
$$\Om_n=(-1)^{pn}H_n(1)=\wtl H_n(1).$$

\begin{lmm}
\label{palyndr2}
We have $\wtl H_n(t)=\wtl H_n(1/t\cL)$.
\end{lmm}
\begin{proof}
We have seen in Remark \ref{symmetric H} that
$$\sum_{\la}\lH_\la(t)T^{\n\la}
=\Exp\left(\sum_{n\ge1}\frac{t}{(1-t)(1-t\cL)}\wtl H_n(t)T^n\right).$$
By Lemma \ref{A property} the left hand side of the above equation is invariant under the change of variables $t\mto1/t\cL$. On the other hand the expression
$$\frac{t}{(1-t)(1-t\cL)}$$
is also invariant under this change of variables.
\end{proof}

\begin{proof}[Proof of Theorem \ref{main2}]
It follows from equations \eqref{eq:B'} and \eqref{eq:Om'mu} that
$$C_{>\mu}=\Exp\left(\frac{\Om'_{>\mu}}{\cL-1}\right).$$
Therefore we have to prove that
$$A_{+\infty}/\Exp\left(\frac{\Om'_{>\mu}}{\cL-1}\right)
=\left(A_{+\infty}/\Exp\left(\frac{\Om'_{\ge-\mu}}{\cL-1}\right)\right)(s\inv)$$
in $\lV\pser{s^{\pm1},\TT}$. It follows from the above discussion that the left hand side and the right hand side are equal as $\TT$-rational functions. The theorem will be proved if we will show that these rational functions are actually contained in the ring $\lV[s^{\pm1}]\pser \TT$.
We have seen that
$$A_{+\infty}=
\sum_{\la}\lH_\la(t)T^{\n\la}
=\Exp\left(\sum_{n\ge1}\frac{t}{(1-t)(1-t\cL)}\wtl H_n(t)T^n\right).$$
Therefore we just have to show that
$$\frac{t}{(1-t)(1-t\cL)}\sum_{n\ge1}\wtl H_n(t)\TT^n-\frac{\Om'_{>\mu}}{\cL-1}$$
is in $\lV[s^{\pm1}]\pser \TT$. It is enough to show this just for $\mu=0$ because
the difference between $\Om'_{>\mu}$ and $\Om'_{>0}$ is obviously in $\lV[s^{\pm1}]\pser \TT$. 
We have seen in Lemma \ref{lmm:Hprb>0 is palindromic} that
$$\Om'_{>0}=\sum_{n\ge1}\Om_nT^n\frac{t(\cL-1)}{(1-t)(1-t\cL)}.$$
Therefore we have to show that
$$\frac{\wtl H_n(t)-\Om_n}{(1-t)(1-t\cL)}$$
is a polynomial in $s^{\pm1}$ (or equivalently in $t^{\pm1}$) for every $n\ge1$.
By our assumptions $\wtl H_n(t)$ is a polynomial in $t^{\pm1}$ and $\Om_n=\wtl H_n(1)$. Therefore
$$\wtl H_n(t)-\Om_n=\wtl H_n(t)-\wtl H_n(1)$$
is divisible by $1-t$. It follows from the invariance of $\wtl H_n(t)$ under the change of variables $t\mto1/t\cL$ that the above expression is also divisible by $1-t\cL$.
\end{proof}

\section{Examples}
Conjecture \ref{E twisted} describes the motive of the moduli space $\lM(L,n,d)$ of twisted Higgs bundles, where $\deg L=2g-2+p$, $p\ge0$ and $n,d$ are coprime. In particular it says that the $E$-polynomial of $\lM(L,n,d)$ should be equal to $$E_{g,n}^{(p)}(u,v):=(uv)^{\dim\lM(L,n,d)/2}H_n^{(p)}(1,u,v),$$
where $\dim\lM(L,n,d)=2((g-1)n^2+p\binom n2+1)$.

Most of the existing literature deals with Hodge (or Poincar\'e) polynomials of $\lM(L,n,d)$. The reformulation of our conjecture for Hodge polynomials is contained in Remark \ref{Hodge twisted}. There we have defined certain functions $H_n'^{(p)}(t,u,v)$ and claimed that the Hodge polynomial of $\lM(L,n,d)$ should be equal to $H_n'^{(p)}(1,-u,-v)$. For simplicity we will work only with Poincar\'e polynomials. By our conjecture the Poincar\'e polynomial (for ordinary cohomologies) of $\lM(L,n,d)$ should be equal to $P_{g,n}^{(p)}(y):=H_n'^{(p)}(1,-y,-y)$.
By the Poincar\'e duality we should have
$$E_{g,n}^{(p)}(-y,-y)=y^{2\dim\lM(L,n,d)}P_{g,n}^{(p)}(y\inv).$$

\subsection{\texorpdfstring{$g=0$}{g=0}}
Our tests show that $E_{g,n}^{(p)}=0$ and $P_{g,n}^{(p)}=0$ for $n\ge2$, $p\le2$. For $p=3$ we have
\begin{align*}
E_{0,2}^{(3)}(u,v)&=1\\
E_{0,3}^{(3)}(u,v)&=uv(uv+1)\\
E_{0,4}^{(3)}(u,v)&=(uv)^3(u^3v^3+u^2v^2+3uv+2)\\
E_{0,5}^{(3)}(u,v)&=(uv)^6(u^6v^6+u^5v^5+3u^4v^4+5u^3v^3+7u^2v^2+9uv+5)
\end{align*}
\begin{align*}
P_{0,2}^{(3)}(y)&=y^{2}+1	\\
P_{0,3}^{(3)}(y)&=3y^{8}+4y^{6}+3y^{4}+y^{2}+1\\
P_{0,4}^{(3)}(y)&=
10y^{18}+20y^{16}+22y^{14}+18y^{12}+13y^{10}+9y^{8}+5y^{6}+3y^{4}+y^{2}+1	\\
P_{0,5}^{(3)}(y)&=
40y^{32}+103y^{30}+154y^{28}+165y^{26}+156y^{24}+131y^{22}+105y^{20}\\
&+77y^{18} +56y^{16}+38y^{14}+26y^{12}+15y^{10}+10y^{8}+5y^{6}+3y^{4}+y^{2}+1	
\end{align*}

For $p=4$ we have
\begin{align*}
E_{0,2}^{(4)}(u,v)&=(uv)(uv+1)\\
E_{0,3}^{(4)}(u,v)&=(uv)^4(u^4v^4+u^3v^3+3u^2v^2+4uv+3)\\
E_{0,4}^{(4)}(u,v)&=(uv)^9(u^9v^9+u^8v^8+3u^7v^7 +5u^6v^6+9u^5v^5+13u^4v^4+18u^3v^3\\ &+22u^2v^2+20uv+10)\\
E_{0,5}^{(4)}(u,v)&=(uv)^{16}(u^{16}v^{16}+u^{15}v^{15}+3u^{14}v^{14}+5u^{13}v^{13}+10u^{12}v^{12}+15u^{11}v^{11}\\ &+26u^{10}v^{10}+38u^9v^9+56u^8v^8+77u^7v^7+105u^6v^6+131u^5v^5\\ &+156u^4v^4+165u^3v^3+154u^2v^2+103uv+40)
\end{align*}
\begin{align*}
P_{0,2}^{(4)}(y)&=y^{2}+1	\\
P_{0,3}^{(4)}(y)&=3y^{8}+4y^{6}+3y^{4}+y^{2}+1\\
P_{0,4}^{(4)}(y)&=
10y^{18}+20y^{16}+22y^{14}+18y^{12}+13y^{10}+9y^{8}+5y^{6}+3y^{4}+y^{2}+1	\\
P_{0,5}^{(4)}(y)&=
40y^{32}+103y^{30}+154y^{28}+165y^{26}+156y^{24}+131y^{22}+105y^{20}\\
&+77y^{18} +56y^{16}+38y^{14}+26y^{12}+15y^{10}+10y^{8}+5y^{6}+3y^{4}+y^{2}+1	
\end{align*}

\begin{rmr}
The Poincar\'e polynomials for $g=0$, $p=4$, and $r\le5$ were computed by Steven Rayan \cite{rayan_geometry}. The corresponding twisted Higgs bundles are the so-called co-Higgs bundles studied in \cite{rayan_co-higgs}. It is checked in \cite{rayan_geometry} that the Poincar\'e polynomials coincide with the above expressions.
\end{rmr}

\subsection{\texorpdfstring{$g=1$}{g=1}}
We have $P_{1,n}^{(0)}(y)=(1+y)^2$ for $n\ge1$.
For $p=1$ we have
\begin{align*}
E_{1,2}^{(1)}(u,v)&=(uv)^{2}(1-u)(1-v)(uv+1)\\
E_{1,3}^{(1)}(u,v)&=(uv)^{4}(1-u)(1-v)(u^{3}v^{3}+u^{2}v^{2}-uv^{2}-u^{2}v+2uv-v-u+1)\\
E_{1,4}^{(1)}(u,v)&=(uv)^{7}(1-u)(1-v)(u^{6}v^{6}+u^{5}v^{5}-u^{4}v^{5}-2u^{3}v^{4}+3u^{4}v^{4}-u^{5}v^{4}\\&-2u^{4}v^{3}-4u^{2}v^{3}+uv^{3}
+5u^{3}v^{3}+v^{2}-5uv^{2}-4u^{3}v^{2}+8u^{2}v^{2}-5u^{2}v\\&+u^{3}v+7uv-3v+2-3u+u^{2})\\
P_{1,2}^{(1)}(y)&=(1+y)^2(y^{2}+1)\\
P_{1,3}^{(1)}(y)&=(1+y)^2(y^{6}+2y^{5}+2y^{4}+2y^{3}+y^{2}+1)\\
P_{1,4}^{(1)}(y)&=(1+y)^2(2y^{12}+6y^{11}+9y^{10}+10y^{9}+10y^{8}+8y^{7}+5y^{6}+4y^{5}+3y^{4}\\&+2y^{3}+y^{2}+1)
\end{align*}

For $p=2$ we have
\begin{align*}
P_{1,2}^{(2)}(y)&=(1+y)^2(y^{4}+2y^{3}+y^{2}+1)\\
P_{1,3}^{(2)}(y)&=(1+y)^2(3y^{12}+8y^{11}+10y^{10}+10y^{9}+10y^{8}+8y^{7}+5y^{6}+4y^{5}+3y^{4}\\
&+2y^{3}+y^{2}+1)\\
P_{1,4}^{(2)}(y)&=(1+y)^2(10y^{24}+40y^{23}+78y^{22}+108y^{21}+131y^{20}+144y^{19}+137y^{18}\\
&+120y^{17}+108y^{16}+94y^{15}+73y^{14}+56y^{13}+46y^{12}+36y^{11}+25y^{10}\\ &+18y^{9}+14y^{8}+10y^{7}+6y^{6}+4y^{5}+3y^{4}+2y^{3}+y^{2}+1)
\end{align*}

\subsection{\texorpdfstring{$g=2$}{g=2}}
For $p=0$ we have
\begin{align*}
P_{2,2}^{(0)}(y)&=(1+y)^4(2y^{6}+4y^{5}+2y^{4}+4y^{3}+y^{2}+1)\\
P_{2,3}^{(0)}(y)&=(1+y)^4(6y^{16}+24y^{15}+48y^{14}+68y^{13}+67y^{12}+64y^{11}+48y^{10}+32y^{9}\\&+29y^{8}+16y^{7}+10y^{6}+8y^{5}+3y^{4}+4y^{3}+y^{2}+1)\\
P_{2,4}^{(0)}(y)&=(1+y)^4(22y^{30}+144y^{29}+456y^{28}+976y^{27}+1554y^{26}+1984y^{25}\\ &+2184y^{24}+2180y^{23}+1991y^{22}+1696y^{21}+1421y^{20}+1140y^{19}+883y^{18}\\ &+672y^{17}+501y^{16}+384y^{15}+269y^{14}+192y^{13}+139y^{12}+96y^{11}+69y^{10}\\ &+40y^{9}+31y^{8}+20y^{7}+11y^{6}+8y^{5}+3y^{4}+4y^{3}+y^{2}+1)
\end{align*}

\begin{rmr}
The Poincar\'e polynomials $P_{g,n}^{(p)}(y)$ (as well as Hodge polynomials) for $p=0$ and $g\ge1$ were conjectured by Hausel and Rodriguez-Villegas, see Section \ref{sec:hrv}. These formulas were checked for $n=2$ using the results of Hitchin \cite{hitchin_self-duality} and for $n=3$ and small $g$ using the results of Gothen \cite{gothen_betti}. Recently this conjecture was also checked for $n=4$ and small $g$ in \cite{garcia-prada_motive}.
\end{rmr}

For $p=1$ we have 
\begin{align*}
P_{2,2}^{(1)}(y)&=(1+y)^4(2y^{8}+4y^{7}+8y^{6}+4y^{5}+2y^{4}+4y^{3}+y^{2}+1)\\
P_{2,3}^{(1)}(y)&=(1+y)^4(6y^{22}+36y^{21}+96y^{20}+168y^{19}+207y^{18}+216y^{17}+210y^{16}\\&+184y^{15}+149y^{14}+120y^{13}+92y^{12}+72y^{11}+49y^{10}+32y^{9}+29y^{8}\\&+16y^{7}+10y^{6}+8y^{5}+3y^{4}+4y^{3}+y^{2}+1)
\end{align*}

For $p=2$ we have
\begin{align*}
P_{2,2}^{(2)}(y)&=(1+y)^4(2y^{10}+8y^{9}+8y^{8}+8y^{7}+8y^{6}+4y^{5}+2y^{4}+4y^{3}+y^{2}+1)\\
P_{2,3}^{(2)}(y)&=(1+y)^4(10y^{28}+64y^{27}+184y^{26}+344y^{25}+477y^{24}+560y^{23}+583y^{22}\\&+560y^{21}+522y^{20}+464y^{19}+386y^{18}+320y^{17}+267y^{16}+208y^{15}\\&+158y^{14}+124y^{13}+93y^{12}+72y^{11}+49y^{10}+32y^{9}+29y^{8}+16y^{7}\\&+10y^{6}+8y^{5}+3y^{4}+4y^{3}+y^{2}+1)
\end{align*}

\begin{rmr}
The above formulas coincide (up to some multipliers and change of signs) with the formulas in \cite[Section 4]{chuang_motivic} obtained by solving recursively the ADHM recursion formula.
\end{rmr}

\bibliography{../tex/papers}

\providecommand{\bysame}{\leavevmode\hbox to3em{\hrulefill}\thinspace}
\providecommand{\href}[2]{#2}
\begin{thebibliography}{10}

\bibitem{behrend_motivica}
Kai Behrend and Ajneet Dhillon, \emph{{O}n the motivic class of the stack of
  bundles}, Adv. Math. \textbf{212} (2007), no.~2, 617--644.

\bibitem{chuang_chamber}
{Wu-Yen} Chuang, {Duiliu-Emanuel} Diaconescu, and Guang Pan, \emph{{C}hamber
  {S}tructure and {W}allcrossing in the {ADHM} {T}heory of {C}urves {II}},
  2009, \eprint{arxiv}{0908.1119}.

\bibitem{chuang_motivic}
\bysame, \emph{{M}otivic {W}allcrossing and {C}ohomology of {T}he {M}oduli
  {S}pace of {H}itchin {P}airs}, 2010, \eprint{arxiv}{1004.4195}.

\bibitem{danilov_newton}
V.~I. Danilov and A.~G. Khovanski{\u\i}, \emph{{N}ewton polyhedra and an
  algorithm for calculating {H}odge-{D}eligne numbers}, Izv. Akad. Nauk SSSR
  Ser. Mat. \textbf{50} (1986), no.~5, 925--945.

\bibitem{deligne_theoriea}
Pierre Deligne, \emph{{T}h\'eorie de {H}odge. {II}}, Inst. Hautes \'Etudes Sci.
  Publ. Math. \textbf{40} (1971), 5--57.

\bibitem{diaconescu_chamber}
{Duiliu-Emanuel} Diaconescu, \emph{{C}hamber {S}tructure and {W}allcrossing in
  {T}he {ADHM} {T}heory of {C}urves {I}}, 2009, \eprint{arxiv}{0904.4451}.

\bibitem{garcia-prada_motive}
Oscar Garcia-Prada, Jochen Heinloth, and Alexander Schmidt, \emph{{O}n the
  motive of moduli of chains and {H}iggs bundles}, in preparation.

\bibitem{getzler_mixed}
Ezra Getzler, \emph{{M}ixed {H}odge structures of configuration spaces},
  Preprint 96-61, Max Planck Institute for Mathematics, Bonn, 1996,
  \eprint{arxiv}{alg-geom/9510018}.

\bibitem{gothen_betti}
Peter~B. Gothen, \emph{{T}he {B}etti numbers of the moduli space of stable rank
  {$3$} {H}iggs bundles on a {R}iemann surface}, Internat. J. Math. \textbf{5}
  (1994), no.~6, 861--875.

\bibitem{gothen_homological}
Peter~B. Gothen and Alastair~D. King, \emph{{H}omological algebra of twisted
  quiver bundles}, J. London Math. Soc. (2) \textbf{71} (2005), no.~1, 85--99,
  \eprint{arxiv}{math/0202033}.

\bibitem{hausel_mirrora}
Tam{\'a}s Hausel, \emph{{M}irror symmetry and {L}anglands duality in the
  non-abelian {H}odge theory of a curve}, Geometric methods in algebra and
  number theory, Progr. Math., vol. 235, Birkh\"auser Boston, Boston, MA, 2005,
  \eprint{arxiv}{math.AG/0406380}, pp.~193--217.

\bibitem{hausel_mixed}
Tam{\'a}s Hausel and Fernando Rodriguez-Villegas, \emph{{M}ixed {H}odge
  polynomials of character varieties}, Invent. Math. \textbf{174} (2008),
  no.~3, 555--624, \eprint{arxiv}{math/0612668}, With an appendix by Nicholas
  M. Katz.

\bibitem{heinloth_note}
Franziska Heinloth, \emph{{A} note on functional equations for zeta functions
  with values in {C}how motives}, Ann. Inst. Fourier (Grenoble) \textbf{57}
  (2007), no.~6, 1927--1945, \eprint{arxiv}{math/0512237}.

\bibitem{hitchin_self-duality}
N.~J. Hitchin, \emph{{T}he self-duality equations on a {R}iemann surface},
  Proc. London Math. Soc. (3) \textbf{55} (1987), no.~1, 59--126.

\bibitem{kapranov_elliptic}
M.~Kapranov, \emph{{T}he elliptic curve in the {S}-duality theory and
  {E}isenstein series for {K}ac-{M}oody groups}, 2000,
  \eprint{arxiv}{math/0001005}.

\bibitem{kontsevich_stability}
Maxim Kontsevich and Yan Soibelman, \emph{{S}tability structures, motivic
  {D}onaldson-{T}homas invariants and cluster transformations}, 2008,
  \eprint{arxiv}{0811.2435}.

\bibitem{kontsevich_cohomological}
\bysame, \emph{{C}ohomological {H}all algebra, exponential {H}odge structures
  and motivic {D}onaldson-{T}homas invariants}, 2010,
  \eprint{arxiv}{1006.2706}.

\bibitem{larsen_rationality}
Michael Larsen and Valery~A. Lunts, \emph{{R}ationality criteria for motivic
  zeta functions}, Compos. Math. \textbf{140} (2004), no.~6, 1537--1560,
  \eprint{arxiv}{math/0212158}.

\bibitem{Macdonald95}
I.~G. Macdonald, \emph{{S}ymmetric functions and {H}all polynomials}, second
  ed., Oxford Mathematical Monographs, The Clarendon Press Oxford University
  Press, New York, 1995, With contributions by A. Zelevinsky, Oxford Science
  Publications.

\bibitem{mozgovoy_computational}
Sergey Mozgovoy, \emph{{A} computational criterion for the {K}ac conjecture},
  J. Algebra \textbf{318} (2007), no.~2, 669--679,
  \eprint{arxiv}{math.RT/0608321}.

\bibitem{mozgovoy_wall-crossing}
\bysame, \emph{{W}all-crossing formulas for framed objects}, 2011,
  \eprint{arxiv}{1104.4335}.

\bibitem{rayan_geometry}
Steven Rayan, \emph{{G}eometry of co-{H}iggs bundles}, PhD thesis, in
  preparation.

\bibitem{rayan_co-higgs}
\bysame, \emph{{C}o-{H}iggs bundles on $\mathbb {P}^1$}, 2010,
  \eprint{arxiv}{1010.2526}.

\end{thebibliography}
\bibliographystyle{../tex/hamsplain}

\end{document}